\documentclass[11pt]{article}
\usepackage[english]{babel}
\usepackage{amsmath,amsthm}
\usepackage[T1]{fontenc}
\usepackage{amsfonts}
\usepackage{mathrsfs}
\usepackage{color}
\usepackage{bbm}
\usepackage{enumerate}
\usepackage{amsrefs}
\usepackage{mathtools}
\usepackage{hyperref}

\newtheorem{theorem}{Theorem}[section]
\newtheorem{corollary}[theorem]{Corollary}
\newtheorem{lemma}[theorem]{Lemma}
\newtheorem{proposition}[theorem]{Proposition}
\newtheorem{assumption}{Assumption}
\theoremstyle{definition}
\newtheorem{definition}[theorem]{Definition}
\theoremstyle{remark}
\newtheorem{remark}[theorem]{Remark}
\newtheorem{example}[theorem]{Example}
\numberwithin{equation}{section}
\begin{document}
\def\Pro{{\mathbb{P}}}
\def\E{{\mathbb{E}}}
\def\e{{\varepsilon}}
\def\veps{{\varepsilon}}
\def\ds{{\displaystyle}}
\def\nat{{\mathbb{N}}}
\def\Dom{{\textnormal{Dom}}}
\def\dist{{\textnormal{dist}}}
\def\R{{\mathbb{R}}}
\def\O{{\mathcal{O}}}
\def\T{{\mathcal{T}}}
\def\Tr{{\textnormal{Tr}}}
\def\sgn{{\textnormal{sign}}}
\def\I{{\mathcal{I}}}
\def\A{{\mathcal{A}}}
\def\H{{\mathcal{H}}}
\def\S{{\mathcal{S}}}

\title{Existence and uniqueness of global solutions to the stochastic heat equation with super-linear drift on an unbounded spatial domain}%
\author{M. Salins\\ Boston University \\ msalins@bu.edu}
\maketitle

\begin{abstract}
  We prove the existence and uniqueness of global solutions to the semilinear stochastic heat equation on an unbounded spatial domain with forcing terms that grow superlinearly and satisfy an Osgood condition $\int 1/|f(u)|du = +\infty$ along with additional restrictions. For example, consider the forcing  $f(u) =  u \log(e^e + |u|)\log(\log(e^e+|u|))$. A new dynamic weighting procedure is introduced to control the solutions, which are unbounded in space.
\end{abstract}

%
%

\section{Introduction} \label{S:intro}
  We prove the existence and uniqueness of global solutions to the semilinear heat equation with additive noise
  \begin{equation} \label{eq:SPDE}
    \frac{\partial u}{\partial t}(t,x)  = \frac{1}{2}\Delta u(t,x) + f(u(t,x)) + \sigma \dot{W}(t,x), x \in \mathbb{R}^d
  \end{equation}
  in the situation where the nonlinearity $f:\mathbb{R} \to \mathbb{R}$ is superlinear, but satisfies the following Osgood condition \cite{o-1898}
  \begin{equation} \label{eq:osgood-intro}
    \int_c^\infty \frac{1}{ |f(x)|} dx = +\infty \text{ for all } c>0.
  \end{equation}

  As a motivating examples consider
  $$f(x) = x \log(e + |x|) \text{ or }f(x) = x \log(e^e+|x|) \log (\log( e^e + |x|)).$$

  In the classical theory of ordinary differential equations (ODE), solutions to the differential equation
  \begin{equation}
    \frac{d\vartheta }{dt}(t,x) = f(\vartheta(t,x)), \ \ \ \vartheta(0) = x
  \end{equation}
  are finite for all $t>0$ if and only if  $f$ satisfies the Osgood condition \eqref{eq:osgood-intro}. In fact, the solution is explicit and
  \begin{equation}
    \vartheta(t,x) = F^{-1}(F(x) + t)
  \end{equation}
  where
  \begin{equation}
    F(y) := \int_c^y \frac{1}{f(y)}dy.
  \end{equation}


  Unlike in the ODE setting, the Osgood condition, without additional assumptions, does not  determine whether the solutions to deterministic partial differential equations (PDEs) blow up. Consider the deterministic elliptic PDE
  \begin{equation} \label{eq:determ-PDE}
    \frac{\partial v}{\partial t}(t,x)  = \frac{1}{2}\Delta v(t,x) + f(v(t,x)).
  \end{equation}
  If we impose the additional restriction that the initial data $u(0,x)$ is uniformly bounded, then the solution $u(t,x)$ never explodes if and only if \eqref{eq:osgood-intro} holds \cite{lrs-2013}. When the initial data of an elliptic PDE is unbounded, however, there exist nonlinearities $f$ that satisfy the Osgood condition \eqref{eq:osgood-intro} and integrable initial data $u(0,x)$, such that the solution blows up in finite time (or maybe instantly) \cite{lrs-2013, lrs-2014}.

  Several authors have investigated the blow-up properties of stochastic heat equations with Osgood forcing on a bounded domain. Bonder and Groisman  \cite{bg-2009} proved that for a one-dimensional stochastic heat equation with additive noise on a bounded spatial interval, the converse of \eqref{eq:osgood-intro} implies that solutions blow up with probability one. Dalang, Khoshnevisan, and Zhang \cite{dkz-2019} also studied the one-dimensional bounded domain setting and showed that global solutions to the stochastic heat equation can exist even when the forcing term $f$ and the multiplicative noise $\sigma$ are both super-linear. Specifically, they demonstrate that if $f$ grows slower than $|x|\log|x|$ and $\sigma$ grows slower that $|x|(\log|x|)^{\frac{1}{4}}$ then solutions do not blow up.

  Foondun and Nualart \cite{fn-2021} strengthen the result of \cite{bg-2009} and prove that on a multidimensional bounded spatial domain, global solutions to the semilinear heat equation with additive noise exist if \eqref{eq:osgood-intro} holds and solutions blow up if \eqref{eq:osgood-intro} fails.  Foondun and Nualart also investigate the unbounded domain setting and proved that global solutions cannot exist if  \eqref{eq:osgood-intro} fails.

  A similar problem has been investigated for the stochastic wave equation on bounded and unbounded spatial domains \cite{ms-2021}, \cite{fn-2020}.

  A recent preprint \cite{sz-2021} provides a partial converse to Foondun and Nualart's result for stochastic heat equations with spatial domain $\mathbb{R}$. They demonstrate that when $f(x)$ grows slower than $C(1+ |x|\log(|x|))$ and there is a bounded multiplicative noise  that solutions to \eqref{eq:SPDE} never explode.

  In this paper, we extend this result to any spatial dimension and a much larger class of forcing terms that satisfy the Osgood condition \eqref{eq:osgood-intro}. Specifically, we show that if $f$ satisfies \eqref{eq:osgood-intro} and some additional restrictions (see Assumption \ref{assum:f} in the following section), then there exists a mild solution to \eqref{eq:SPDE} that never blows up. This setup allow nonlinearities that grow much faster than $|x|\log(|x|)$ considered in the recent preprint \cite{sz-2021}. Based on the results \cite{lrs-2013,lrs-2014} for the blow-up of deterministic PDEs with Osgood forcing, the Osgood condition \eqref{eq:osgood-intro} on its own should not be sufficient to imply global existence of the stochastic heat equation, but we leave the construction of a counterexample for future work.

  The proof of the existence result is based on a dynamic weighting procedure. A mild solution to \eqref{eq:SPDE}  is the solution to the integral equation
  \begin{align} \label{eq:mild}
    u(t,x) = &\int_{\mathbb{R}^d} G(t,x-y)u(0,y)dy + \int_0^t \int_{\mathbb{R}^d} G(t-s,x-y)f(u(s,y))dyds \nonumber \\
    &+ \sigma \int_0^t \int_{\mathbb{R}^d} G(t-s,x-y)W(dyds).
  \end{align}
  In the above equation
  \begin{equation} \label{eq:heat-kernel}
    G(t,x):= (2\pi t)^{-\frac{d}{2}} \exp \left(-\frac{|x|^2}{2t} \right)
  \end{equation}
  is the Gaussian heat kernel.
  The stochastic integral $Z(t,x) = \sigma \int_0^t \int_{\mathbb{R}^d} G(t-s,x-y)W(dyds)$ is a Gaussian process that is stationary in space. By \cite{qw-1973}, $Z(t,x)$ is unbounded in $x$ for any $t>0$, but for any finite $T>0$,
  \begin{equation}
    \sup_{t \in [0,T]}\sup_{x \in \mathbb{R}^d}\frac{|Z(t,x)|}{\rho_0(x)} <+\infty \text{ with probability one}
  \end{equation}
  where the weight
  \begin{equation}
    \rho_0(x) := \sqrt{\log(e + |x|^2)}.
  \end{equation}

  If the forcing term $f$ grows no faster than linearly, then a Gr\"onwall argument proves that
  \begin{equation} \label{eq:rho0-bound-intro}
    \sup_{t \in [0,T]}\sup_{x \in \mathbb{R}^d}\frac{|u(t,x)|}{\rho_0(x)} <+\infty \text{ with probability one}
  \end{equation}
  as well. In the presence of a superlinear forcing term, like $f(x) = x \log|x|$, however, \eqref{eq:rho0-bound-intro} will fail to be true. Intuitively, this is due to the fact that a superlinear forcing causes unbounded solutions to stretch in space as well as time.

  To deal with the superlinear forcing we build a dynamic weight that satisfies
  \begin{equation} \label{eq:dynamic-weight-intro}
    \begin{cases}
      \frac{\partial \rho}{\partial t} (t,x) \geq \frac{1}{2}\Delta \rho(t,x) + h_\alpha(\rho(t,x))\\
      \rho(0,x) = \rho_0(x).
    \end{cases}
  \end{equation}
  In the above inequality, $h(x)$ is an increasing, convex function that dominates $f$ and satisfies \eqref{eq:osgood-intro}. For appropriate $\alpha>1$, $h_{\alpha}(x) = \frac{h(x^\alpha)}{x^{\alpha-1}}$. The function $h_\alpha$ grows asymptotically faster than both $h$ and $f$ but still satisfies the Osgood condition \eqref{eq:osgood-intro}. With this appropriate choice of dynamic weighting, we will be able to prove, via a sequence of approximations, that a mild solution $u(t,x)$ solving \eqref{eq:mild} exists and that for any finite time horizon $T>0$,
  \begin{equation}
    \sup_{t \in [0,T]} \sup_{x \in \mathbb{R}^d} \frac{|u(t,x)|}{\rho(t,x)}< +\infty \text{ with probability one}
  \end{equation}
  implying the existence of a global solution.

  In general, weak (in the PDE sense) solutions to \eqref{eq:SPDE} on unbounded spatial domains are not unique. In the simplest case of a deterministic heat equation ($f \equiv 0$ and $\sigma \equiv 0$), Tychonoff \cite{t-1935} proved that solutions are not unique, but that solutions are unique in the class of functions that grow slower than $Me^{c|x|^2}$ in space. We show that, under an appropriate assumption on the forcing terms, that weak solutions to \eqref{eq:SPDE} are unique in the class of random fields that grow slower than $e^{|x|^\nu}$ for some $\nu \in (0,2)$.

  In Section \ref{S:assum} we outline the assumptions on the forcing term, the noise, and the initial data, and state the main results. In Section \ref{S:examples}, we give examples of forces that satisfy our assumptions and an example of a force that satisfies the Osgood condition \eqref{eq:osgood-intro}, but grows too quickly. In Section \ref{S:stoch-conv} we recall known results about the growth rate of the stochastic convolution. In Section \ref{S:weight} we introduce the dynamic weight. In Section \ref{S:determ-forcing}, we use the dynamic weighting procedure to prove the existence of mild solutions to an associated problem with deterministic forcing. The proof of the existence result is in Section \ref{S:proof}. The proof of uniqueness is in Section \ref{S:uniq}.
\section{Assumptions and the result} \label{S:assum}
\begin{assumption} \label{assum:f}
  \begin{enumerate}
  \item The drift term $f:\mathbb{R} \to \mathbb{R}$ is locally Lipschitz continuous. Specifically, there exists an increasing function $L: [0,+\infty) \to [0,+\infty)$ such that for all $|u_1|\leq R, |u_2|\leq R$,
      \begin{equation} \label{eq:locally-Lip}
        |f(u_1) - f(u_2)| \leq L(R) |u_1 - u_2|.
      \end{equation}
  \item There exists a function $h: [0,+\infty) \to (0,+\infty)$ that is strictly positive, twice-differentiable, increasing, and strictly convex that satisfies the Osgood condition
  \begin{equation} \label{eq:osgood}
    \int_0^\infty \frac{1}{h(x)}dx = +\infty,
  \end{equation}
  grows no faster than polynomially,  meaning that there exist $C>0$ and $p>1$ such that
  \begin{equation} \label{eq:h-polynomial-growth}
    h(x) \leq C (1 + |x|^p),
  \end{equation}
  and dominates $f$ in the sense that
  \begin{equation} \label{eq:dominate}
    |f(x)| \leq h(|x|).
  \end{equation}
  \item  For any $t>0$ there exists $\nu_t \in (0,2)$ and $C_t>0$ such that for any $x>0$,
  \begin{equation}  \label{eq:growth-restriction}
    H^{-1}(H(x) + t) \leq C_t \exp(\exp(|x|^{\nu_t}))
  \end{equation}
  where
  \begin{equation}
    H(x) = \int_0^x \frac{1}{h(y)}dy.
  \end{equation}
  \item In addition to $h$ being convex, we assume that for $\alpha >1$,
  \begin{equation} \label{eq:h-alpha-assum}
    h_\alpha(x) := \frac{h(x^\alpha)}{x^{\alpha-1}}
  \end{equation}
  is convex for $x\geq1$.
  \end{enumerate}
\end{assumption}

\begin{remark}
  The function $\vartheta(t,x) = H^{-1}(H(x) + t)$ is the unique solution to the ordinary differential equation where $x$ represents the initial data
  \begin{equation} \label{eq:ode}
    \frac{d \vartheta}{dt}(t,x) = h(\vartheta(t,x)),\ \ \  \vartheta(0,x)=x.
  \end{equation}
  The restriction on the growth rate of $\vartheta$ in  \eqref{eq:growth-restriction} is a technical assumption that is required for the dynamic weighting method used in section \ref{S:determ-forcing} to work properly. We expect that the Osgood condition, on its own, does not imply the existence of global solutions without imposing a growth rate like \eqref{eq:growth-restriction}.
\end{remark}

\begin{remark}
  The condition that $h_\alpha$ defined in \eqref{eq:h-alpha-assum} is convex is very natural. $x \mapsto h(x)$ grows superlinearly so $\frac{h(x^\alpha)}{x^{\alpha-1}}$ should also grow superlinearly. This restriction removes pathological cases from consideration.
\end{remark}

\begin{assumption} \label{assum:noise}
  The Gaussian noise $\dot{W}$ is white in time and spatially homogeneous (see \cite{walsh-book}). The covariance is formally described by
  \begin{equation} \label{eq:covariance}
    \E \dot{W}(t,x)\dot{W}(s,y) =  \delta(t-s) \Lambda(x-y)
  \end{equation}
  where $\delta$ is the Dirac delta measure and $\Lambda$ is a positive and positive definite generalized function whose Fourier transform $\mu = \mathcal{F}(\Lambda)$ satisfies the strong Dalang assumption (see \cite{ss-2002}). Namely, there exists $\eta\in (0,1)$ such that
  \begin{equation} \label{eq:strong-dalang}
    \int_{\mathbb{R}^d} \frac{\mu(d\xi)}{1 + |\xi|^{2(1-\eta)}} < +\infty
  \end{equation}
\end{assumption}

\begin{assumption} \label{assum:initial-data}
  The initial data $x \mapsto u(0,x)$ is continuous and
  \begin{equation}
    \sup_{x \in \mathbb{R}^d} \frac{|u(0,x)|}{\sqrt{\log(e + |x|^2)}} <+\infty.
  \end{equation}
\end{assumption}

\begin{definition}
  A mild solution $u(t,x)$ is called a global solution if $u(t,x)$ solves \eqref{eq:mild} and $u(t,x)$ is finite for all $(t,x) \in [0,+\infty)\times \mathbb{R}^d$ with probability one.
\end{definition}

The main results of the paper are Theorem \ref{thm:main} and Theorem \ref{thm:uniq}.

\begin{theorem} \label{thm:main}
  Under Assumptions \ref{assum:f}, \ref{assum:noise}, and \ref{assum:initial-data}, there exists a global mild solution solving \eqref{eq:mild}.
\end{theorem}
The proof of Theorem \ref{thm:main} is in Section \ref{S:proof}.

We introduce the following assumption that implies a uniqueness result.

\begin{assumption} \label{assum:Lip}
  In addition to Assumption \ref{assum:f}, assume that there exists $\e \in (0,1)$ such that the local Lipschitz constants $L(R)$ defined in \eqref{eq:locally-Lip} satisfy
  \begin{equation} \label{eq:Lip-over-log}
    \lim_{R \to +\infty} \frac{L(R)}{(\log(R))^{1 + \e} } = 0.
  \end{equation}
\end{assumption}
\begin{remark}
  Assumption \ref{assum:Lip} is very reasonable. Because of the Osgood condition \eqref{eq:osgood}, for any $\e>0$,
  \begin{equation}
    \liminf_{|x| \to \infty} \frac{|f(x)|}{|x| (\log|x|)^{1 + \e}} =0.
  \end{equation}
  If the $\liminf$ were positive, then $\int_c^\infty \frac{1}{|f(x)|} <+\infty$ for some $c>0$.
  Assumption \ref{assum:Lip} more or less requires that the $\limsup$ is also $0$, while imposing some reasonable local Lipschitz continuity behavior as well.
\end{remark}

\begin{definition}
  A random field $u(t,x)$ is a weak solution to \eqref{eq:SPDE} if for any twice-differentiable test function $\phi: \mathbb{R}^d \to \mathbb{R}$ with compact support,
  \begin{align} \label{eq:weak}
    &\int_{\mathbb{R}^d} u(t,x)\phi(x)dx - \int_{\mathbb{R}^d} u(0,x)\phi(x)dx  \nonumber\\
    &=\int_0^t \int_{\mathbb{R}^d} u(s,x) \frac{1}{2}\Delta \phi(x)dxds + \int_0^t \int_{\mathbb{R}^d} f(u(s,x))\phi(x)dxds \nonumber\\
    &+ \int_0^t \int_{\mathbb{R}^d}\sigma\phi(x)W(dxds).
  \end{align}
\end{definition}

\begin{theorem} \label{thm:uniq}
  Assume Assumptions \ref{assum:f}, \ref{assum:noise},  \ref{assum:initial-data}, and \ref{assum:Lip}.
  Let $\nu \in (0,2/(1+\e))$ and $T>0$. There exists at most one weak solution to $u(t,x)$ to \eqref{eq:SPDE} for $t \in [0,T]$ that satisfies the growth condition
  \begin{equation}
    \Pro \left(\sup_{t \in [0,T]} \sup_{x \in \mathbb{R}^d} |u(t,x)|e^{-|x|^\nu}< +\infty \right)=1.
  \end{equation}
\end{theorem}
The proof of Theorem \ref{thm:uniq} is in Section \ref{S:uniq}. Notice that  our theorem does not claim uniqueness for the $\nu=2$ case, which is attainable when $f$ is globally Lipschitz continuous.

In Section \ref{S:examples}, we will give many examples of superlineat forcing terms, $f$, that satisfy Assumption \ref{assum:Lip} for arbitrarily small $\e \in (0,1)$. If this is the case then the following corollary holds.
\begin{corollary}
  Assume Assumptions \ref{assum:f}, \ref{assum:noise},  \ref{assum:initial-data}, and \ref{assum:Lip}. If Assumption \ref{assum:Lip} holds for arbitrarily small $\e \in (0,1)$, then for any $\nu \in (0,2)$, there exists at most one weak solution to $u(t,x)$ to \eqref{eq:SPDE} for $t \in [0,T]$ that satisfies the growth condition
  \begin{equation}
    \Pro \left(\sup_{t \in [0,T]} \sup_{x \in \mathbb{R}^d} |u(t,x)|e^{-|x|^\nu}< +\infty \right)=1.
  \end{equation}
\end{corollary}

\section{Examples} \label{S:examples}
Before proving the main results, we demonstrate that there exist many examples of superliner forcing terms that satisfy Assumptions \ref{assum:f} and \ref{assum:Lip}. Furthermore, we show that there exist forcing terms that satisfy Osgood condition \eqref{eq:osgood} but violate \eqref{eq:growth-restriction}.


\begin{example}
Define
\[L_1(x) = \log(x),\]
\[L_{n+1}(x) = \log(L_n(x))\]
and
\[E_1(x) = \exp(x),\]
\[E_{n+1}(x) = \exp(E_n (x)).\]

For any $n \in \mathbb{N}$, define
\begin{equation} \label{eq:h_n}
  h_n(x) = (x + E_n(1))\prod_{k=1}^{n-1} L_k(x+E_n(1)).
\end{equation}
We add the $E_n(1)$ constants just to ensure that everything is well-posed and non-negative.

For example,
\begin{align}
  &h_2(x) = (x + E_2(1))\log(x+E_2(1)) \nonumber \\
  &h_3(x) = (x+ E_3(1)) \log (x + E_3(1)) \log \log (x+E_3(1)), \nonumber\\
  &h_4(x) = (x+ E_4(1)) \log (x + E_4(1)) \log \log (x+E_4(1)) \log \log \log (x + E_4(1)). \nonumber
\end{align}

Direct calculations confirm that
\[H_n(x) := \int_{0}^x \frac{1}{h_n(y)}dy = L_n(x + E_n(1)) - 1. \]
and
\begin{equation} \label{eq:vartheta}
  \vartheta_n(t,x) := H_n^{-1}(H_n(x) + t) = E_n(L_n(x + E_n(1)) + t) - E_n(1).
\end{equation}

To prove that the growth restriction \eqref{eq:growth-restriction} holds, we observe that for any fixed $t>0$ and $n \geq 2$,

\[\lim_{x \to \infty} \frac{L_n(x) + t}{L_{n-1}(x)} = \lim_{x \to \infty} \frac{\log(L_{n-1}(x)) + t}{L_{n-1}(x)} = 0.\]
In particular, this means that for any $t>0$  there exists $R_t>0$ such that
for all $x> R_t$,
\[L_n(x) + t\leq L_{n-1}(x)\]
and therefore,
\[E_n(L_n(x) + t) \leq E_n (L_{n-1}(x)) \leq \exp(x) \text{ for } x>R_t.\]
For any $\nu>0$,
\[\lim_{x \to \infty} \frac{\exp(x)}{\exp(\exp(x^\nu))} = 0.\]
So for any fixed $n\in \mathbb{N}$ and $t>0$, there exists $C_{t,\nu}>0$ such that
\[H_n^{-1}(H_n(x) + t) \leq C_{\nu,t}\exp(\exp(x^\nu)).\]

Therefore, each of $h_n$ defined in \eqref{eq:h_n} satisfies \eqref{eq:growth-restriction} with arbitrarily small $\nu_t>0$. So our theory allows for superlinear terms that grow much faster than $h_n$ for all $n$.

Furthermore,  we can check that for each $n$, the first derivative is bounded by
\begin{equation}
  \frac{d h_n}{dx}(x) \leq n \frac{h_n(x)}{x} \leq n \prod_{k=1}^{n-1} L_k(x + E_n(1)).
\end{equation}
Therefore, each $h_n$ satisfies Assumption \ref{assum:Lip} for arbitrarily small $\e>0$.

\end{example}

Unfortunately, the dynamic weighting method used to prove the main result of this paper cannot work for every function satisfying the Osgood condition $\int_0^\infty \frac{1}{h(x)} =\infty$ as the following example demonstrates because the Osgood condition does not imply \eqref{eq:growth-restriction}.
\begin{example}
  Let $x_n$ be a sequence of numbers such that
  \[x_1 = 1 \text{ and } x_{n+1} = \exp (\exp(\exp(x_n)))\]
  Let $H(x)$ be a positive, increasing, smooth, concave function such that
  \[H(x_n) = n \]
  Let $h(x) = \frac{1}{H'(x)}$. By these definitions, when $t=1$
  \begin{equation}
   H^{-1}(H(x_n) + 1)= x_{n+1} = \exp(\exp(\exp(x_n))).
  \end{equation}
  This proves that $H^{-1}(H(x_n) + 1) >> \exp( \exp(|x_n|^{\nu})) $ for any $\nu\in (0,2)$. Therefore, \eqref{eq:growth-restriction} cannot hold.

\end{example}

\section{Stochastic convolution} \label{S:stoch-conv}
The stochastic convolution term of the mild solution \eqref{eq:mild} is defined to be
\begin{equation} \label{eq:stoch-conv}
  Z(t,x) = \sigma\int_0^t \int_{\mathbb{R}^d} G(t-s,x-y)W(dyds).
\end{equation}

Sanz-Sol\'e and Sarr\`a \cite{ss-2002} proved that the strong Dalang assumption (our Assumption \ref{assum:noise}) implies that $Z(t,x)$ is almost surely H\"older continuous in $t$ and $x$.

For any fixed $t>0$, $x \mapsto Z(t,x)$ is a stationary Gaussian process. By Theorems 3.1 and 3.2 of \cite{qw-1973} (see also Theorem 1.1 of \cite{k-2019} for the case of the fractional heat equation),
\begin{equation}
  \Pro \left(\sup_{x\in \mathbb{R}^d}  \frac{ |Z(t,x)| }{\sqrt{\log(e + |x|)}} <+\infty \right) = 1.
\end{equation}

Additionally, we can modify the results of \cite{qw-1973} to show that for any fixed time horizon $T>0$,
\begin{equation} \label{eq:stoch-conv-grow}
  \Pro \left(\sup_{t \in [0,T]} \sup_{x \in \mathbb{R}^d}\frac{|Z(t,x)|}{\sqrt{\log(e + |x|)}}<+\infty  \right) = 1
\end{equation}

These results inspire the definition of the dynamic weight $\rho(t,x)$ in the following section.

\section{Function spaces and weights} \label{S:weight}
Even though our existence result (Theorem \ref{thm:main}) claims that the solutions $u(t,x)$ to \eqref{eq:mild} are global and cannot blow up, these random fields are unbounded in $x$ almost surely for any $t>0$. We will prove that the solutions are global by introducing an appropriate weight $\rho(t,x)$ such that the quotient
\begin{equation}
   \frac{u(t,x)}{\rho(t,x)}
\end{equation}
has the property that for any fixed time horizon $T>0$,
\[\sup_{t \in [0,T]}\sup_{x \in \mathbb{R}^d} \left|\frac{u(t,x)}{\rho(t,x)}\right| < +\infty \text{ with probability one.}\]

Define the $C_0 = C_0(\mathbb{R}^d)$ to be the Banach space of continuous functions $\phi:\mathbb{R}^d \to \mathbb{R}$ such that
\begin{equation} \label{eq:C_0-def}
  \lim_{|x| \to +\infty} |\phi(x)|  = 0.
\end{equation}
$C_0(\mathbb{R}^d)$ is endowed with the supremum norm
\begin{equation} \label{eq:C_0-norm}
  |\phi|_{C_0} := \sup_{x \in \mathbb{R}^d} |\phi(x)|.
\end{equation}

Now we define the weights used in the proof of the main result.

\begin{lemma} \label{lem:h-alpha}
  For any $\alpha>1$ the function $h_\alpha: [0,+\infty) \to [0,+\infty)$ defined in \eqref{eq:h-alpha-assum} is strictly positive
  \begin{equation} \label{eq:h_alpha-bounded-below}
    \inf_{x>0} h_\alpha(x)>0
  \end{equation}
  and $h_\alpha$ satisfies the Osgood condition
  \begin{equation} \label{eq:h_alpha-osgood}
    \int_0^\infty \frac{1}{h_\alpha(x)}dx = +\infty.
  \end{equation}
\end{lemma}
\begin{proof}
The positivity of $h_\alpha$ follows from the fact that $h$ is positive and continuous.

If we let
\begin{equation} \label{eq:H-alpha}
  H_\alpha(x): = \int_{0}^x \frac{1}{h_\alpha(y)}dy = \int_0^x \frac{y^{\alpha-1}}{h(y^\alpha)}dy,
\end{equation}
then direct calculations show that
\begin{equation}  \label{eq:H-alpha-explicit}
  H_\alpha(x) = \frac{1}{\alpha}H(x^\alpha),
\end{equation}
which diverges as $x \to +\infty$ because of \eqref{eq:osgood}.

\end{proof}
%
%

An important consequence of the convexity of $h$ is that the value of $h$ of a product of positive numbers can be bounded by $h_\alpha$ and $h_{\frac{\alpha}{\alpha-1}}$ as the following lemma demonstrates.
\begin{lemma} \label{lem:convexity}
  Let $h:[0,+\infty) \to [0,+\infty)$ be a convex function and let $\alpha>1$. For any numbers $\rho>0$ and $q>0$,
  \begin{equation} \label{eq:convex-ineq}
    h(\rho q) \leq q h_\alpha(\rho) + \rho h_{\frac{\alpha}{\alpha-1}}(q)
  \end{equation}
\end{lemma}

\begin{proof}
  First we observe that for any real numbers $q>0$, $\rho>0$,
  \begin{align}
    q \rho
    = \left(\frac{q\rho - q^{\frac{\alpha}{\alpha-1}}}{ \rho^\alpha - q^{\frac{\alpha}{\alpha-1}}} \right)\rho^\alpha + \left(\frac{ \rho^\alpha - q \rho}{\rho^\alpha - q^{\frac{\alpha}{\alpha-1}} }\right) q^{\frac{\alpha}{\alpha-1}}.
  \end{align}
  It follows from the convexity of $h$ that
  \begin{equation} \label{eq:convex-bound}
    h(q \rho) \leq  \left(\frac{q\rho - q^{\frac{\alpha}{\alpha-1}}}{ \rho^\alpha - q^{\frac{\alpha}{\alpha-1}}}\right) h(\rho^\alpha) + \left(\frac{ \rho^\alpha - q \rho}{\rho^\alpha - q^{\frac{\alpha}{\alpha-1}} } \right) h(q^{\frac{\alpha}{\alpha-1}}).
  \end{equation}
  The prefactors can be bounded
  \begin{align} \label{eq:prefactor-bound-1}
    &\frac{q\rho - q^{\frac{\alpha}{\alpha-1}}}{ \rho^\alpha - q^{\frac{\alpha}{\alpha-1}}} 
    =  \frac{q}{\rho^{\alpha -1}} \left(\frac{1 - \left(\frac{q^{\frac{1}{\alpha-1}}}{\rho} \right)}{1 - \left(\frac{q^{\frac{1}{\alpha-1}}}{\rho} \right)^\alpha} \right)
    \leq \frac{q}{\rho^{\alpha -1}}
  \end{align}
  and
  \begin{align} \label{eq:prefactor-bound-2}
    &\frac{ \rho^\alpha - q \rho}{\rho^\alpha - q^{\frac{\alpha}{\alpha-1}} } 
    = \frac{\rho}{q^{\frac{1}{\alpha -1}}} \left(\frac{\left(\frac{\rho^{\alpha-1}}{q} \right) -1 }{\left(\frac{\rho^{\alpha-1}}{q} \right)^{\frac{\alpha}{\alpha-1}} -1} \right)
    \leq  \frac{\rho}{q^{\frac{1}{\alpha -1}}}
  \end{align}
  because the positive real-valued functions $\frac{1 - x}{1-x^\alpha} \in [0,1]$ and  $ \frac{x -1}{x^{\frac{\alpha}{\alpha-1}} -1} \in [0,1] $ for all $x \geq 0$.

  Combining \eqref{eq:convex-bound}, \eqref{eq:prefactor-bound-1}, and \eqref{eq:prefactor-bound-2},
  \begin{equation}
    h(q \rho) \leq \frac{q h(\rho^\alpha)}{\rho^{\alpha-1}} + \frac{\rho h(q^{\frac{\alpha}{\alpha-1}})}{q^{\frac{1}{\alpha - 1}}} = q h_\alpha(\rho) + \rho h_{\frac{\alpha}{\alpha-1}}(q).
  \end{equation}
\end{proof}

Now we can define our dynamic weight function. Let $\rho_0: \mathbb{R}^d \to [0,+\infty)$ be defined by
\begin{equation} \label{eq:rho-0}
  \rho_0(x) = \sqrt{\log(e + |x|^2)}.
\end{equation}
This is a twice-differentiable weight that grows at the same rate as the stochastic convolution (see  \eqref{eq:stoch-conv-grow}).
\begin{definition} \label{def:rho}
Define the dynamic weight function $\rho: [0,T]\times \mathbb{R}^d \to [0,+\infty)$ by
\begin{equation} \label{eq:weight-def}
  \rho(t,x) = \E \left( H_\alpha^{-1} \left( H_\alpha( \rho_0(x + B(t))) + t \right)\right).
\end{equation}
where $B(t)$ is a standard $d$-dimensional Wiener process, $\rho_0$ is defined in \eqref{eq:rho-0}, and $H_\alpha$ is defined by \eqref{eq:H-alpha}.
\end{definition}

\begin{proposition} \label{prop:super-solution}
  If $\alpha \in \left(1, 2/\nu_{2T} \right)$, then the weight $\rho(t,x)$ defined in \eqref{eq:weight-def} is finite and differentiable for any $t \in [0,T]$ and $x \in \mathbb{R}^d$ and satisfies
  \begin{equation}
   \begin{cases}
    \displaystyle {\frac{\partial \rho}{\partial t}(t,x) \geq \frac{1}{2}\Delta \rho(t,x) + h_\alpha(\rho(t,x)),}\\
    \rho(0,x)= \rho_0(x) = \sqrt{\log \left(e + |x|^2\right) }.
   \end{cases}
  \end{equation}
\end{proposition}
\begin{proof}
  The proof is a straightforward application of Ito formula and Jensen's inequality.

  From \eqref{eq:H-alpha-explicit},
  \begin{equation}
    H_\alpha^{-1}(H_\alpha(\rho_0(x)) + t) = \left(H^{-1} \left(H((\rho_0(x))^\alpha \right) + \alpha t \right)^{\frac{1}{\alpha}}.
  \end{equation}
  By assumption  \eqref{eq:growth-restriction} and definition \eqref{eq:rho-0} there exists some $C_T>0$ such that for all $t \in [0,T]$,
  \begin{equation}
    H_\alpha^{-1}(H_\alpha(\rho_0(x)) + t) \leq C_{T} \exp\left(\frac{1}{\alpha} \exp \left( \left(\log(e + |x|^2) \right)^{\alpha (\nu_{\alpha T})/2} \right) \right).
  \end{equation}
  Because $\alpha$ was chosen so that $\frac{\alpha \nu_{\alpha T}}{2} <1$,
  \begin{equation} \label{eq:H-exp}
    H_\alpha^{-1}(H_\alpha(\rho_0(x)) + t) \leq C_{T,\alpha} e^{|x|}.
  \end{equation}
  Therefore,
  \begin{equation} \label{eq:rho-exp}
    \rho(t,x) = \E H_\alpha^{-1}(H_\alpha(\rho_0(x + B(t))) + t) \leq C_{T,\alpha} \E e^{|x|}e^{|B(t)|} \leq C_{T,\alpha}e^{|x| + t},
  \end{equation}
  proving that $\rho(t,x)$ is well defined.

  Also, from the definition of $H_\alpha$ \eqref{eq:H-alpha}, it is clear that
  \begin{equation} \label{eq:partial-deriv-H-alpha}
    \left(H_\alpha^{-1} \right)'(t) = \frac{d }{dt} H_\alpha^{-1}(t) = h_\alpha(H_\alpha^{-1}(t)).
  \end{equation}

  By Ito formula,
  \begin{align}
    &\frac{\partial \rho}{\partial t}(t,x) \nonumber\\
    &= \E \left(\left(H_\alpha^{-1} \right)' \left(H_\alpha(\rho_0(x + B(t))) + t \right)\right) + \frac{1}{2}\Delta \rho(t,x)\nonumber \\
    &= \E  h_\alpha \left( H_\alpha^{-1}(H_\alpha(\rho_0(x + B(t))) + t) \right) + \frac{1}{2}\Delta \rho(t,x).
  \end{align}
  Notice that the use of Ito formula is justified and the above terms are integrable because of the polynomial growth assumption \eqref{eq:h-polynomial-growth}, and the exponential growth \eqref{eq:H-exp}.
  Finally, by Jensen's inequality and the assumption that $h_\alpha(x) $  is convex \eqref{eq:h-alpha-assum},
  \begin{equation} \label{eq:weight-deriv}
    \frac{\partial \rho}{\partial t}(t,x) \geq \frac{1}{2}\Delta \rho(t,x) + h_\alpha(\rho(t,x)).
  \end{equation}
\end{proof}

Lastly, we show that, without loss of generality, for any $t>0$, $\rho(t,x)$ dominates $\rho_0(x)$ as $|x| \to + \infty$.
\begin{lemma} \label{lem:rho-dominates}
  Without loss of generality, we can choose $h$ satisfying Assumption 1 such that for any $t>0$,
  \begin{equation}
    \liminf_{|x| \to \infty} \frac{\rho(t,x)}{\rho_0(x)} =+\infty.
  \end{equation}
\end{lemma}
\begin{proof}
  For any fixed $t>0$,
  \begin{equation} \label{eq:adding-W-neglig}
    \lim_{|x| \to \infty} \frac{\rho_0(x + B(t))}{\rho_0(x)} = 1 \text{ with probability one}.
  \end{equation}
  Without loss of generality, we can choose $h$ in such a way that $h(y) \geq c |y| \log(|y|)$ for some $c>0$ (we can simply add $|y|\log|y|$ to the original $h$). Consequently, for $y>0$,
  \begin{equation}
    H^{-1}(H(y) + t) \geq y^{e^{ct}}.
  \end{equation}
  and for any $y>0$,
  \begin{equation}
    \lim_{y \to +\infty} \frac{H_\alpha^{-1}(H_\alpha(y) + t)}{y} = +\infty.
  \end{equation}
  Therefore, by \eqref{eq:adding-W-neglig} it follows that with probability one,
  \begin{align}
    &\liminf_{|x| \to \infty} \frac{{H}_\alpha^{-1}(H_\alpha(\rho_0(x + B(t))) + t)}{\rho_0(x)}\nonumber\\
    &\geq \liminf_{|x| \to \infty} \frac{{H}_\alpha^{-1}(H_\alpha(\rho_0(x + B(t))) + t)}{\rho_0(x + B(t))}\nonumber\\\nonumber\\
     &\geq +\infty.
  \end{align}
  By Fatou's lemma,
  \begin{equation}
    \liminf_{|x| \to \infty} \frac{\rho(t,x)}{\rho_0(x)} = \liminf_{|x| \to \infty} \E \left( \frac{{H}_\alpha^{-1}( H_\alpha(\rho_0(x + B(t))) + t)}{\rho_0(x)} \right) \geq +\infty.
  \end{equation}
\end{proof}

\section{Deterministic forcing} \label{S:determ-forcing}
Given a deterministic function $z: [0,T] \times \mathbb{R}^d \to \mathbb{R}$ that is continuous and satisfies for any $T>0$,
\begin{equation}
  \sup_{t \in [0,T]} \sup_{x \in \mathbb{R}^d} \frac{|z(t,x)|}{\rho_0(x)} < +\infty,
\end{equation}
we prove that the solution to the integral equation
\begin{equation} \label{eq:u-int-eq}
  u(t,x) = \int_0^t \int_{\mathbb{R}^d}G(t-s,x-y) f(u(s,y))dyds  + z(t,x)
\end{equation}
exists and does not blow up. Later we will use the results of this section to prove Theorem \ref{thm:main} by replacing $z(t,x)$ pathwise with
\begin{equation}
  z(t,x) = \int_{\mathbb{R}^d} G(t,x-y)u(0,y)dy + \sigma \int_0^t \int_{\mathbb{R}^d} G(t-s,x-y)W(dyds).
\end{equation}

We begin by building approximations. For $N>0$ define the cutoff version of $z$ by
\begin{equation} \label{eq:z_N}
    z_N(t,x) =
    \begin{cases}
      z(t,x) & \text{ if } |z(t,x)|\leq N,\\
      N & \text{ if } z(t,x)>N,\\
      -N & \text{ if } z(t,x) < -N.
    \end{cases}
\end{equation}
Similarly, we build an approximating sequence to $f$. Let
\begin{equation}
  f_N(u) =
  \begin{cases}
    f(u) & \text{ if } |u|\leq N\\
    f(N) & \text{ if } u>N\\
    f(-N) & \text{ if } u < -N.
  \end{cases}
\end{equation}
 Each of these $f_N$ are globally Lipschitz continuous by Assumption \ref{assum:f}.

Because of the boundedness of $z_N$ and the Lipschitz continuity of $f_N$, standard Picard  iteration arguments using the supremum norm  prove that there exists a unique bounded solution to $u_N$ solving
\begin{equation} \label{eq:u_N-eq}
  u_N(t,x) = \int_0^t \int_{\mathbb{R}^d}G(t-s,x-y) f_N(u_N(s,y))dyds + z_N(t,x).
\end{equation}

Now we show that we can use dynamic weighting techniques to get a bound on $u_N$ that is independent of $N$.

  Let $u_N(t,x)$ be a solution to \eqref{eq:u_N-eq}.
  Define
  \begin{align*}
    &v_N(t,x) := u_N(t,x) - z_N(t,x) \\
    &= \int_0^t \int_{\mathbb{R}^d} G(t-s,x-y) f_N(v_N(s,y) + z_N(s,y))dyds.
  \end{align*}
  Notice that $v_N$ is a weak solution of the PDE
  \begin{equation} \label{eq:PDE}
    \frac{\partial v_N}{\partial t}(t,x)  = \frac{1}{2} \Delta v_N(t,x) + f_N(v_N(t,x) + z_N(t,x)).
  \end{equation}
  In general, $v_N(t,x)$ is not strongly differentiable in $t$ or $x$. See, for example, Chapter 4.3 of \cite{pazy}.
  Let $\rho(t,x)$ be the weight defined in \eqref{eq:weight-def}, let
  \begin{equation}
    M:=\sup_{x \in \mathbb{R}^d} \sup_{t \in [0,T]} \frac{|z(t,x)|}{\rho_0(x)} < +\infty,
  \end{equation}
  and define the weighted function
  \begin{equation} \label{eq:q_N-def}
    q_N(t,x) := \frac{v_N(t,x) + M \rho_{0}(x) }{\rho(t,x)} .
  \end{equation}

  Note that
  \begin{equation}
    q_N(t,x) \geq \frac{v_N(t,x) + z_N(t,x)}{\rho(t,x)} = \frac{u_N(t,x)}{\rho(t,x)},
  \end{equation}
  and $q_N(t,x)$ is weakly differentiable, while $\frac{u_N(t,x)}{\rho(t,x)}$ is not weakly differentiable.

The next result shows that a standard technique (See Proposition 6.2.2 of \cite{cerrai-book} or Theorem 7.7 of \cite{dpz-book}) used to regularize solutions in the bounded domain setting can also be applied to this setting of weighted spaces on $\mathbb{R}^d$.
\begin{lemma} \label{lem:derivative}
  Without loss of generality, we can assume that $q_N$ is strongly differentiable
  and that
  \begin{align} \label{eq:q_n-deriv}
    &\frac{\partial q_N}{\partial t}(t,x)\nonumber\\
    &\leq \frac{1}{2} \Delta q_N(t,x) +  \frac{\nabla q_N(t,x)\cdot \nabla \rho(t,x)}{\rho(t,x)}- \frac{\frac{M}{2} \Delta\rho_{0}(x)}{\rho(t,x)}
     \nonumber \\
    &\qquad+ \frac{f_N(q_N(t,x)\rho(t,x) - M \rho_0(x) + z_N(t,x))}{\rho(t,x)} - \frac{q_N(t,x)h_\alpha(\rho(t,x)) }{\rho(t,x)}.
  \end{align}

  More specifically, there exists a sequence $q_{N,\lambda}$ that is strongly differentiable for which
  \begin{equation}
    \lim_{\lambda \to +\infty} \sup_{x \in \mathbb{R}^d} \left| q_{N,\lambda}(t,x) - q_N(t,x)\right| = 0 \text{ for all } t>0
  \end{equation}
  and for any $T>0$,
  \begin{align*}
    \limsup_{\lambda \to \infty} \int_0^T \sup_{x \in \mathbb{R}^d}\Bigg( &\frac{\partial q_{N,\lambda}}{\partial t}(t,x) - \frac{1}{2} \Delta q_{N,\lambda}(t,x)- \frac{\nabla q_{N,\lambda}(t,x)\cdot \nabla \rho(t,x)}{\rho(t,x)} \nonumber\\
     &+ \frac{\frac{M}{2} \Delta\rho_{0}(x)}{\rho(t,x)}
     - \frac{f_N(q_{N,\lambda}(t,x)\rho(t,x)-M\rho_0(x) + z_N(t,x))}{\rho(t,x)} \nonumber\\
    &+ \frac{q_{N,\lambda}(t,x)h_\alpha(\rho(t,x)) }{\rho(t,x)} \Bigg)dt \leq 0.
  \end{align*}
\end{lemma}

\begin{proof}
%
The full details of the construction of the approximating sequence $q_{N,\lambda}$ via resolvent operators is in Proposition \ref{prop:WLOG-strong} and Corollary \ref{cor:WLOG-strong} in the appendix.

By Corollary \ref{cor:WLOG-strong}, the approximation $q_{N,\lambda} = \frac{v_{N,\lambda}(t,x)}{\rho(t,x)} + \frac{M \rho_0(x)}{\rho(t,x)}$ is strongly differentiable,
\begin{equation}
  \lim_{\lambda \to \infty} \sup_{x \in \mathbb{R}^d} |q_{N,\lambda}(t,x)-q_N(t,x)| = 0,
\end{equation}
and
\begin{align} \label{eq:q_n-deriv}
  &\frac{\partial q_{N,\lambda}}{\partial t}(t,x)\nonumber\\
    &= \frac{1}{2} \Delta q_{N,\lambda}(t,x) +  \frac{\nabla q_{N,\lambda}(t,x)\cdot \nabla \rho(t,x)}{\rho(t,x)}  + q_{N,\lambda}(t,x)\frac{\left(\frac{1}{2}\Delta \rho(t,x) -\frac{\partial \rho}{\partial t}(t,x) \right) }{\rho(t,x)}\nonumber \\
    &\qquad- \frac{\frac{M}{2} \Delta\rho_{0}(x)}{\rho(t,x)}  + \frac{f_N(q_{N,\lambda}(t,x)\rho(t,x) - M\rho_0(x) + z_N(t,x))}{\rho(t,x)}\nonumber\\
    &\qquad+ \frac{\delta_\lambda(t,x)}{\rho(t,x)}.
\end{align}
where
\begin{equation}
  \lim_{\lambda \to +\infty} \int_0^T\sup_{x \in \mathbb{R}^d} \frac{|\delta_\lambda(t,x)|}{\rho(t,x)}dt = 0.
\end{equation}
By \eqref{eq:weight-deriv}
\begin{align} \label{eq:q_n-deriv}
  &\frac{\partial q_{N,\lambda}}{\partial t}(t,x)\nonumber\\
    &\leq \frac{1}{2} \Delta q_{N,\lambda}(t,x) +  \frac{\nabla q_{N,\lambda}(t,x)\cdot \nabla \rho(t,x)}{\rho(t,x)}  - q_{N,\lambda}(t,x)\frac{h_\alpha(\rho(t,x)) }{\rho(t,x)}\nonumber \\
    &\qquad- \frac{\frac{M}{2} \Delta\rho_{0}(x)}{\rho(t,x)}  + \frac{f_N(q_{N,\lambda}(t,x)\rho(t,x) - M\rho_0(x) + z_N(t,x))}{\rho(t,x)}\nonumber\\
    &\qquad+ \frac{\delta_\lambda(t,x)}{\rho(t,x)}.
\end{align}

This proves the result.
\end{proof}

\begin{lemma} \label{lem:determ-existence-bounded}
  Assume that $z: [0,T]\times \mathbb{R}^d \to \mathbb{R}$ is a continuous function such that
  \begin{equation}
    M:= \sup_{x \in \mathbb{R}^d} \sup_{t \in [0,T]} \frac{|z(t,x)|}{\rho_0(x)}  < +\infty.
  \end{equation}
  where $\rho_0(x) = \sqrt{\log(e + |x|^2)}$ as defined in \eqref{eq:rho-0}.
  Let $\rho(t,x)$ be the dynamic weight described in \eqref{eq:weight-def}. Let $u_N$ be the unique solution to \eqref{eq:u_N-eq}.
  There exists a constant $K(T,M)>0$ (independent of $N$) such that for any $N>0$,
  \begin{equation}
    \sup_{t \in [0,T]} \sup_{x \in \mathbb{R}^d} \frac{|u_N(t,x)|}{\rho(t,x)} \leq K(T,M).
  \end{equation}
\end{lemma}

\begin{proof}
  By Lemma \ref{lem:derivative}, we assume without loss of generality that $q_N$ defined in \eqref{eq:q_N-def} is strongly differentiable.

  Because $v_N(t,x) $ is bounded and $\lim_{|x| \to +\infty} \frac{\rho_0(x)}{\rho(t,x)} = 0$ by Lemma \ref{lem:rho-dominates},
  \begin{equation}
    \lim_{|x| \to \infty} q_N(t,x) = \lim_{|x| \to \infty} \frac{v_N(t,x) + M \rho_0(x) }{\rho(t,x)} = 0.
  \end{equation}
  Therefore, $q_N(t,\cdot)$ is $C_0(\mathbb{R}^d)$, the space of continuous functions that disappear at infinity.

  This means that any maximizer of $q_N(t,x)$ is attained. Therefore, there exists $x_t \in \mathbb{R}^d$ such that
  \begin{equation}
    \sup_{x \in \mathbb{R}^d} q_{N}(t,x) = q_{N}(t,x_t).
  \end{equation}
  The left derivative is bounded by (see Proposition \ref{prop:norm-diff} and Lemma \ref{lem:derivative}),
  \begin{align}
    &\frac{d^-}{dt} \sup_{x \in \mathbb{R}^d} q_{N}(t,x) \nonumber\\
    &\leq \frac{1}{2}\Delta q_{N}(t,x_t) + \frac{\nabla q_N(t,x_t) \cdot \nabla \rho(t,x_t)}{\rho(t,x_t)} \nonumber\\
    & \qquad+ \frac{f_N( q_{N}(t,x_t)\rho(t,x_t) +z_N(t,x_t) - M \rho_{0}(x_t))}{\rho(t,x_t)} \nonumber \\
    &\qquad - \frac{M}{2}\frac{ \Delta \rho_{0}(x_t)}{\rho(t,x_t)}- \frac{q_{N}(t,x_t) h_\alpha(\rho(t,x_t))}{\rho(t,x_t)}.
  \end{align}
  Because $x_t$ is a local maximizer and we assumed without loss of generality that $q_N$ is twice differentiable,
  \begin{align*}
    &\frac{1}{2}\Delta q_N(t,x_t) \leq 0\text{ and } \nabla q_N(t,x_t) = 0.
  \end{align*}
  Therefore,

  \begin{align} \label{eq:q-deriv-1}
    &\frac{d^-}{dt} \sup_{x \in \mathbb{R}^d} q_{N}(t,x) \nonumber\\
    &\leq  -\frac{M \Delta \rho_0(x_t)}{2\rho(t,x)}\nonumber\\
    &\qquad +\frac{f_N( q_{N}(t,x_t)\rho(t,x_t) +z_N(t,x_t) - M \rho_0(x_t) )}{\rho(t,x_t)} - \frac{q_{N}(t,x_t) h_\alpha(\rho(t,x_t))}{\rho(t,x_t)}.
  \end{align}


  %
%
%

  If $q_N(t,x_t) > 2M$, then by Lemma \ref{lem:rho-dominates},
  \begin{equation}
    q_N(t,x_t)\rho(t,x_t) + z_N(t,x_t) - M \rho_0(x_t)>0
  \end{equation}
  By Assumption \ref{assum:f}, $h$ is increasing and  $f_N(x) \leq h(x)$ for $x>0$, and because $|z_N(t,x)| \leq M \rho_{0}(x)$, whenever $q_N(t,x_t) > 2M$,
  \[f_N\left(q_{N}(t,x_t) \rho(t,x_t) - M \rho_0(x_t)  + z_N(t,x_t)\right) \leq h(q_{N}(t,x_t) \rho(t,x_t)).\]

  Therefore, by \eqref{eq:q-deriv-1} and the fact that $\sup_{x \in \mathbb{R}^d} |\Delta \rho_0(x)|<+\infty$,
  \begin{align}
    \frac{d^-}{dt} \max\left\{\sup_{x \in \mathbb{R}^d} q_{N}(t,x) , 2M \right\} \leq &\frac{h(q_{N}(t,x_t)\rho(t,x_t))}{\rho(t,x_t)} - \frac{q_{N}(t,x_t) h_\alpha(\rho(t,x_t))}{\rho(t,x_t)} \nonumber\\
    &+CM
  \end{align}

  By Lemma \ref{lem:convexity},
  \begin{equation}
    \frac{d^-}{dt} \max\left\{\sup_{x \in \mathbb{R}^d} q_{N}(t,x) , 2M \right\} \leq CM + h_{\frac{\alpha}{\alpha-1}}(q_{N}(t,x_t))
  \end{equation}
  Let $H_{\frac{\alpha}{\alpha-1}}$ be defined as in \eqref{eq:H-alpha}. Then
  \begin{align}
    \frac{d^-}{dt} H_{\frac{\alpha}{\alpha-1}} \left(\max\left\{\sup_{x \in \mathbb{R}^d} q_{N}(t,x) , 2M \right\}\right)
    \leq CM + 1 .
  \end{align}
  The above line holds because $H_{\frac{\alpha}{\alpha-1}}'(x) = \frac{1}{h_{\frac{\alpha}{\alpha-1}}(x)}$, which is bounded from above. Consequently,
  \begin{align} \label{eq:q_N-bound}
    &\max\left\{\sup_{x \in \mathbb{R}^d} q_{N}(t,x) , 2M \right\} \leq H_{\frac{\alpha}{\alpha-1}}^{-1} \left( H_{{\frac{\alpha}{\alpha-1}}} \left(2M \right) + (CM+1)t  \right).
  \end{align}
  and
  \begin{equation}
    \sup_{x \in \mathbb{R}^d} q_N(t,x) \leq \max\left\{2M, H_{\frac{\alpha}{\alpha-1}}^{-1} \left( H_{{\frac{\alpha}{\alpha-1}}} \left(M \right) + (CM+1)t  \right) \right\}
  \end{equation}

%
  In particular, for any $x \in \mathbb{R}^d$ and $t \in [0,T]$,
  \begin{align} \label{eq:uniform-uN-bound-signed}
    u_N(t,x) &= v_N(t,x) + z(t,x) \nonumber\\
    & \leq  v_N(t,x) + M\rho_0(x) \nonumber \\
    &\leq q_N(t,x)\rho(t,x) \nonumber\\
    & \leq K(M,T) \rho(t,x)
  \end{align}
  with $K(M,T) = \max\left\{2M,H_{\frac{\alpha}{\alpha-1}}^{-1} \left( H_{{\frac{\alpha}{\alpha-1}}} \left(2M \right) + (CM+1)T  \right)\right\} $.

  So far, we have only proven an upper bound on $u_N(t,x)$. We can get the lower bound using the same argument as above with
  \begin{equation}
    \tilde q_N(t,x) = -\frac{u_N(t,x)}{\rho(t,x)} + \frac{M\rho_0(x) }{\rho(t,x)}
  \end{equation}
  The preceding argument requires  no modification and \eqref{eq:q_N-bound} holds for $\tilde q_N$ as well.
  Therefore,
  \begin{equation} \label{eq:uniform-uN-bound}
    |u_N(t,x)| \leq K(M,T) \rho(t,x).
  \end{equation}
\end{proof}

\begin{theorem} \label{thm:existence}
  Let $z:[0,+\infty)\times \mathbb{R}^d$ be a deterministic function satisfying for any $T>0$,
  \begin{equation}
    \sup_{t \in [0,T]} \sup_{x \in \mathbb{R}^d}\frac{|z(t,x)|}{\rho_0(x)}< +\infty
  \end{equation}
  Then there exists a global solution solving
  \begin{equation}
    u(t,x) = \int_0^t \int_{\mathbb{R}^d} G(t-s,x-y)f(u(s,y))dyds + z(t,x)
  \end{equation}
\end{theorem}
\begin{proof}
  Let $z_N$ and  $u_N$ be the sequences defined in \eqref{eq:z_N} and \eqref{eq:u_N-eq}. Let
  \begin{equation}
    v_N(t,x) = u_N(t,x) - z_N(t,x) = \int_0^t \int_{\mathbb{R}^d} G(t-s,x-y)f(u(s,y))dyds.
  \end{equation}

  In Assumption \ref{assum:f}, we assumed that $H^{-1}(H(x) + t) \leq \exp(\exp(|x|^{\nu_T}))$ for $t \in [0,T]$.
  Therefore, by the definition of $\rho$, \eqref{eq:weight-def}
  \begin{align} \label{eq:rho-bound}
    &|\rho(t,x)| \leq \E \exp \left( \frac{1}{\alpha}\exp \left( \log\left(e+|x + B(t)|^2\right)^{\frac{\alpha \nu_{\alpha T}}{2}}\right) \right) \nonumber\\
    &\leq C_{T,\alpha} \E e^{C_{T,\alpha}|x + B(t)|} \leq C_{T,\alpha} e^{C_{T,\alpha}(|x|+t)}
  \end{align}
  because $\alpha \nu_{2T} < 2$.
  By assumption \eqref{eq:h-polynomial-growth}, $h(x) \leq C (1 +  |x|^p)$. By \eqref{eq:uniform-uN-bound} and \eqref{eq:rho-bound}, there exist large enough constants such that
  \begin{equation} \label{eq:hu-bound}
    | f_N(u_N(t,x))| \leq \left|h(u_N(t,x))\right| \leq C_{T,\alpha,p}e^{C_{T,\alpha,p}(|x|+t)}.
  \end{equation}

  Therefore, by \eqref{eq:uniform-uN-bound} and the fact that
  \begin{equation}
    v_N(t,x) = \int_0^t \int_{\mathbb{R}^d} G(t-s,x-y) f_N(u_N(s,y))dyds
  \end{equation}
  we can use standard regularization properties of the convolution with the fundamental solution of the heat equation along with the exponential bound of \eqref{eq:hu-bound} to show that the $v_N$ are uniformly bounded and equicontinuous on each compact subset of $[0,T]\times \mathbb{R}^d$.
  By the Arzela-Ascoli theorem, there exists a subsequence of $N$ such that
  $\{v_N(t,x)\}$ converges uniformly on bounded subsets of $(x,t)$ to a limit $\tilde{v}(t,x)$. By construction \eqref{eq:z_N}, the $z_N$ converge to $z$. By the dominated convergence theorem, the limit $\tilde{v}$ will solve
  \begin{equation}
    \tilde{v}(t,x) = \int_0^t \int_{\mathbb{R}^d}G(t-s,x-y) f(\tilde{v}(s,y) + z(s,y))ds dy.
  \end{equation}

  Then $\tilde{u}(t,x) = \tilde{v}(t,x) + z(t,x)$ is a solution to
  \begin{equation}
    \tilde{u}(t,x) = \int_0^t \int_{\mathbb{R}^d} G(t-s,x-y)f(\tilde{u}(s,y))dyds + z(t,x)
  \end{equation}
  and satisfies \eqref{eq:uniform-uN-bound} and therefore is
  is a global solution.
\end{proof}

\section{Existence of mild solution -- Proof of Theorem \ref{thm:main}} \label{S:proof}

Let
\begin{equation}
  U_0(t,x) = \int_{\mathbb{R}^d} G(t,x-y)u(0,y)dy.
\end{equation}
Let
\begin{equation}
  Z(t,x) = \sigma \int_0^t \int_{\mathbb{R}^d} G(t-s,x-y)dyds
\end{equation}
 be the stochastic convolution.
By Assumption \ref{assum:initial-data} and Jensen's inequality,
\begin{align}
  U_0(t,x) &= \int_{\mathbb{R}^d} G(t,x-y)u(0,y)dy  \nonumber\\
  &\leq C \int_{\mathbb{R}^d} G(t,x-y)\sqrt{\log(e + |y|^2)}dy\nonumber\\
  &\leq C \sqrt{ \log \left(e + \int_{\mathbb{R}^d} G(t,x-y)|y|^2dy \right)}\nonumber\\
  & \leq C \sqrt{ \log (e + |x|^2 + t)}.
\end{align}
Therefore,
\begin{equation}
  \sup_{t \in [0,T]} \sup_{x \in \mathbb{R}^d} \frac{U_0(t,x)}{\rho_0(x)} \leq C\sup_{t \in [0,T]} \sup_{x \in \mathbb{R}^d}\frac{\sqrt{ \log (e + |x|^2 + t)}}{\sqrt{ \log (e + |x|^2 )}} <+\infty.
\end{equation}
By \eqref{eq:stoch-conv-grow}, $z(t,x) = U_0(t,x) + Z(t,x)$ satisfies the assumptions of \ref{thm:existence} with probability one. Therefore, applying Theorem \ref{thm:existence} pathwise, there exists a random field solving
\begin{align}
  u(t,x) = &\int_{\mathbb{R}^d} G(t,x-y) u(0,y) + \int_0^t \int_{\mathbb{R}^d} G(t-s,x-y)f(u(s,y))dyds \nonumber\\
  &+ \sigma\int_0^t \int_{\mathbb{R}^d} G(t-s,x-y)W(dyds)
\end{align}
and such that for any $T>0$,
\begin{equation}
  \Pro \left( \sup_{t \in [0,T]} \sup_{x \in \mathbb{R}^d} \frac{|u(t,x)|}{\rho(t,x)}<  +\infty \right) =1.
\end{equation}

\section{Uniqueness of the weak solution -- Proof of Theorem \ref{thm:uniq}} \label{S:uniq}
Let $\e$ be from Assumption \ref{assum:Lip} and let $\nu \in (0,2/(1+\e))$ and $T>0$. Assume that $u_1(t,x)$ and $u_2(t,x)$ both weakly solve \eqref{eq:SPDE} and that
\begin{equation} \label{eq:bound-on-u_i}
  \Pro \left( \sup_{t \in [0,T]} \sup_{x \in \mathbb{R}^d} |u_i(t,x)|e^{-|x|^\nu}<+\infty \right) = 1, \ \ \ i \in \{1,2\}.
\end{equation}
Define the random variable
\begin{equation}
  M(\omega):= \max_{i \in \{1,2\}} \sup_{t \in [0,T]} \sup_{x \in \mathbb{R}^d} |u_i(t,x)|e^{-|x|^\nu}
\end{equation}
so that
\begin{equation} \label{eq:u-growth}
  u_i(t,x) \leq M e^{|x|^\nu} \ \ \ \text{ for } i \in \{1,2\}.
\end{equation}
 By the local Lipschitz continuity of $f$ and Assumption \ref{assum:Lip}, for any $t \in [0,T]$ and $x\in \mathbb{R}^d$ (and $\omega \in \Omega$),
\begin{align}
  &|f(u_1(t,x) - f(u_2(t,x))| \leq L(M e^{|x|^\nu}) |u_1(t,x) - u_2(t,x)|\nonumber\\
   &\leq C(1 + (\log(M) + |x|^\nu)^{1 + \e}) |u_1(t,x) - u_2(t,x)|.
\end{align}
Let $\nu_1:= \nu(1 + \e)$ so that by \eqref{eq:Lip-over-log}, the above expression can be written as
\begin{equation} \label{eq:f-bound}
  |f(u_1(t,x) - f(u_2(t,x))| \leq C(1 + |x|^{\nu_1}) |u_1(t,x) - u_2(t,x)|.
\end{equation}
where $C$ is a (random) number that depends on $M$ but does not depend on $t$ or $x$.
Choose $\nu_2 \in (\nu_1, 2)$. These constants are chosen so that $\nu < \nu_1 < \nu_2 < 2$.

For a constant $K>0$ to be defined later, define
\begin{equation} \label{eq:weighted-diff}
  \tilde{q}(t,x):= (u_1(t,x) - u_2(t,x))\exp\left(-(K + |x|^2)^{\frac{\nu_2}{2}}\right).
\end{equation}
Observe that because $\nu_2>\nu$ and because of \eqref{eq:bound-on-u_i},
\begin{equation}
  \int_{\mathbb{R}^d} |\tilde{q}(t,x)|^2dx< +\infty
\end{equation}
for all $t \in [0,T]$.

Notice that the exponent $(K + |x|^2)^{\frac{\nu_2}{2}}$ grows like $|x|^{\nu_2}$, but is twice differentiable in $x$. Because of the additive noise in \eqref{eq:SPDE}, the difference $u_1(t,x) - u_2(t,x)$ and the weighted difference $\tilde q$ are weakly differentiable in $t$ and $x$. The time derivative is (weakly)
\begin{align*}
  \frac{\partial \tilde q}{\partial t}(t,x) = \exp &\left(-(K + |x|^2)^{\frac{\nu_2}{2}} \right)
  \Bigg(\frac{1}{2}\Delta (u_1-u_2)(t,x)
  +(f(u_1(t,x)) - f(u_2(t,x)))  \Bigg)
\end{align*}
Let $\rho(x) := \exp \left((K + |x|^2)^{\frac{\nu_2}{2}} \right)$. This is different from the weight defined in the proof of existence.

Define the unbounded linear operator on $L^2(\mathbb{R}^d)$
\begin{equation}
  \mathscr{A} \phi(x) = \frac{1}{2} \frac{\Delta(\phi \rho)(x)}{\rho(x)} - \nu_2^2 (K + |x|^2)^{\nu_2 -1}\phi(x)
\end{equation}
The extra $- \nu_2^2 (K + |x|^2)^{\nu_2 -1}\phi(x)$ is added to ensure that $\mathscr{A}$ is dissipative in $L^2(\mathbb{R}^2)$ (see \cite[Chapter 1.4]{pazy}). Using this notation, $\tilde{q}$ weakly solves
\begin{equation} \label{eq:tilde-q-weak}
  \frac{\partial \tilde q}{\partial t}(t,x) = \mathscr{A} \tilde q(t,x) + \frac{f(u_1(t,x)) - f(u_2(t,x))}{\rho(x)} + \nu_2^2 (K+|x|^2)^{\nu_2 -1}\tilde{q}(t,x).
\end{equation}

The following lemma collects a result about integrating this operator $\mathscr{A}$ by parts. We present this without proof.
\begin{lemma} \label{lem:int-by-parts-general}
  Suppose that $\phi: \mathbb{R}^d \to \mathbb{R}$ is twice continuously differentiable with compact support. Suppose that $\rho$ and $\psi$ are twice continuously differentiable. Then
  \begin{align} \label{eq:int-by-parts-general}
    &\int_{\mathbb{R}^d} \Delta( \phi \rho)(x)\phi(x) \psi(x)dx \nonumber\\
    &= -\int_{\mathbb{R}^d} |\nabla \phi(x)|^2 \rho(x) \psi(x)dx  + \int_{\mathbb{R}^d} |\phi(x)|^2 \left(\frac{1}{2} \Delta(\rho \psi) -  \nabla \rho(x) \cdot \nabla \psi(x)\right)dx
  \end{align}
\end{lemma}

\begin{lemma} \label{lem:A-dissip}
  The operator $\mathscr{A}: D(\mathscr{A}) \subset L^2(\mathbb{R}^d) \to L^2(\mathbb{R}^d)$ is dissipative (see \cite[Chapter 1.4]{pazy} for a definition). Specifically, for any twice continuously differentiable $\phi$ with compact support,
  \begin{equation} \label{eq:A-dissip}
    \int_{\mathbb{R}^d} \mathscr{A} \phi(x) \phi(x)dx \leq 0.
  \end{equation}
\end{lemma}
\begin{proof}
 If $\phi$ is twice differentiable with compact support, and we apply Lemma \ref{eq:int-by-parts-general} with $\rho(x) = e^{(K+|x|^2)^{\nu_2/2}}$ and $\psi(x) = e^{-(K+|x|^2)^{\nu_2/2}}$, then
\begin{align} \label{eq:int-by-parts}
  &\int_{\mathbb{R}^d} \mathscr{A} \phi(x) \phi(x)dx  \nonumber\\
  &=  -\frac{1}{2}\int_{\mathbb{R}^d} |\nabla \phi(x)|^2 \rho(x) \psi(x)dx  + \frac{1}{2}\int_{\mathbb{R}^d} |\phi(x)|^2 (\frac{1}{2} \Delta(\rho \psi) -  \nabla \rho(x) \cdot \nabla \psi(x))dx\nonumber\\
  & - \nu_2^2\int_{\mathbb{R}^d} (K+|x|^2)^{\frac{\nu_2}{2}}|\phi(x)|^2dx
\end{align}
Direct calculations show that
\begin{equation}
  -\nabla \rho(x)\cdot \nabla \psi(x) = \nu_2^2|x|^2(K+|x|^2)^{\nu_2 -2} \leq \nu_2^2 (K+|x|^2)^{\nu_2 -1},
\end{equation}
\begin{equation}
  \rho(x)\psi(x) \equiv 1,
\end{equation}
and
\begin{equation}
  \Delta(\rho \psi) = \Delta 1 = 0.
\end{equation}
Therefore
\begin{equation} \label{eq:negative-type}
  \int_{\mathbb{R}^d} \mathscr{A} \phi(x) \phi(x)dx \leq - \frac{1}{2}\int_{\mathbb{R}^d} |\nabla \phi(x)|^2 dx \leq 0.
\end{equation}
\end{proof}

While we cannot guarantee that $\tilde{q}$ is a strong solution to \eqref{eq:tilde-q-weak}, we can regularize the process using resolvent operators.
Define the resolvent operators $R(\lambda) = ( \lambda I - \mathscr{A})^{-1}$. Because $\mathscr{A}$ is dissipative, $R(\lambda)$ are bounded linear operators on $L^2(\mathbb{R}^d)$ (see Theorem 1.4.2 of \cite{pazy}) with the properties that (see Chapter 1.3 of \cite{pazy}) for any $\phi \in L^2(\mathbb{R}^d)$,
\begin{align}
  &|R(\lambda)\phi|_{L^2(\mathbb{R}^d)} \leq \frac{1}{\lambda}|\phi|_{L^2(\mathbb{R}^d)}.
  &\lim_{\lambda \to \infty} |\lambda R(\lambda) \phi - \phi|_{L^2(\mathbb{R}^d)} = 0.
\end{align}
Define
\begin{equation}
  \tilde{q}_\lambda(t,x) = \lambda R(\lambda) q(t,\cdot)(x).
\end{equation}
This approximation is a strong solution to
\begin{equation} \label{eq:tilde-q-strong}
  \frac{\partial \tilde q_\lambda}{\partial t}(t,x) = \mathscr{A} \tilde q_\lambda(t,x) + F_\lambda(t,x)
\end{equation}
where
\begin{equation}
  F(t,x) = (f(u_1(t,x)) - f(u_2(t,x)))e^{-(K+|x|^2)^{\nu_2/2}} +\nu_2^2(K+|x|^2)^{\nu_2 -1}\tilde{q}(t,x)
\end{equation}
and
\begin{equation}
  F_\lambda(t,x) = \lambda (R(\lambda) F(t,\cdot))(x)
\end{equation}
Notice that for any $t>0$, $F(t,\cdot) \in L^2(\mathbb{R}^d)$ because \eqref{eq:f-bound}, \eqref{eq:bound-on-u_i} and the choice of $\nu< \nu_1 < \nu_2$,
\begin{equation} \label{eq:F-bound}
  |F(t,x)| \leq C \left((1+|x|^{\nu_1}) + (K+|x|^2)^{\nu_2 - 1} \right) \exp \left( - (K+|x|^2)^{\nu_2/2} +|x|^\nu\right).
\end{equation}
Furthermore,
\begin{align} \label{eq:resolvent-converge1}
  &\lim_{\lambda \to +\infty}  \int_{\mathbb{R}^d} |\tilde{q}(t,x) - \tilde{q}_\lambda(t,x)|^2dx = 0 \ \ \text{ and }\\
  &\lim_{\lambda \to +\infty}  \int_{\mathbb{R}^d} |F(t,x) - F_\lambda(t,x)|^2dx =0.\label{eq:resolvent-converge2}
\end{align}

Finally, we multiply $|\tilde{q}_\lambda(t,x)|^2$ by $\exp \left( -t (K+|x|^2)^{\frac{\nu_2}{2}} \right)$ use the fact that $\tilde{q}_\lambda$ is a strong solution to \eqref{eq:tilde-q-strong} to calculate that
\begin{align} \label{eq:weighted-q-lambda-pde}
  &\frac{d}{dt} \frac{1}{2} \int_{\mathbb{R}^d} |\tilde{q}_\lambda(t,x)|^2 \exp\left( -t (K+|x|^2)^{\nu_2/2} \right)dx \nonumber\\
  &= \int_{\mathbb{R}^d} \mathscr{A}\tilde{q}_\lambda(t,x) \tilde{q}_\lambda(t,x)\exp\left( -t (K+|x|^2)^{\nu_2/2} \right)dx\nonumber\\
  &\qquad + \int_{\mathbb{R}^d} F_\lambda(t,x)\tilde{q}_\lambda(t,x)\exp\left( -t (K+|x|^2)^{\nu_2/2} \right)dx \nonumber\\
  &\qquad - \frac{1}{2} \int_{\mathbb{R}^d} (K+|x|^2)^{\nu_2/2} |\tilde{q}_\lambda(t,x)|^2\exp\left( -t (K+|x|^2)^{\nu_2/2} \right)dx \nonumber\\
  &=: I_1(t) + I_2(t)  + I_3(t).
\end{align}
The most difficult term to analyze is $I_1(t)$. By the definition of $\mathscr{A}$ and Lemma \ref{lem:int-by-parts-general} with $\rho(x) = \exp\left( - (K+|x|^2)^{\nu_2/2}\right)$ and $\psi(t,x) = \exp\left( -(1+t) (K+|x|^2)^{\nu_2/2}\right)$
\begin{align} \label{eq:I1}
  &I_1(t) = \int_{\mathbb{R}^d} \mathscr{A} \tilde{q}_\lambda(t,x) \tilde{q}_\lambda(t,x)\exp\left( -t (K+|x|^2)^{\nu_2/2}\right)\nonumber\\
  &= \int_{\mathbb{R}^d} \frac{1}{2}\Delta (\tilde{q}_\lambda(t,x) \rho(x)) \tilde{q}_\lambda(t,x)\psi(t,x) \nonumber\\
    &\qquad- \frac{1}{2}\int_{\mathbb{R}^d} \nu_2^2 (K+|x|^2)^{\frac{\nu_2}{2}-1} |\tilde{q_\lambda}(t,x)|^2 \exp\left( -t (K+|x|^2)^{\nu_2/2} \right) \nonumber\\
  &\leq -\frac{1}{2}\int_{\mathbb{R}^d} |\nabla \tilde{q}_\lambda(t,x)|^2 \exp\left(-t(K+1|x|^2)^{\frac{\nu_2}{2}}\right)dx \nonumber \\
        &\qquad+\frac{1}{2} \int_{\mathbb{R}^d} |\tilde {q}_\lambda(t,x)|^2  \Bigg((t + t^2/2) \nu_2^2 (K+|x|^2)^{\nu_2-1} \nonumber\\
        &\qquad\qquad\qquad+ t\nu(\nu+d-2)(K+|x|^2)^{\frac{\nu_2}{2}-1}\Bigg)\exp\left(-t(K+1|x|^2)^{\frac{\nu_2}{2}}\right)dx.
\end{align}
Now, for $I_1$, $I_2$, and $I_3$, we use \eqref{eq:resolvent-converge1}--\eqref{eq:resolvent-converge2} to see that
\begin{align}
  &\limsup_{\lambda \to \infty} I_1(t) \nonumber\\
  &\leq
  \frac{1}{2} \int_{\mathbb{R}^d} |\tilde {q}(t,x)|^2  \Bigg((t + t^2/2) \nu_2^2 (K+|x|^2)^{\nu_2-1} \nonumber\\
  &\qquad\qquad+ t\nu(\nu+d-2)(K+|x|^2)^{\frac{\nu_2}{2}-1}\Bigg)\exp\left(-t(K+1|x|^2)^{\frac{\nu_2}{2}}\right)dx.
\end{align}
\begin{equation}
  \limsup_{\lambda \to \infty} I_2(t) = \int_{\mathbb{R}^d} F(t,x)\tilde{q}(t,x)\exp\left( -t (K+|x|^2)^{\nu_2/2} \right)dx
\end{equation}
and
\begin{equation}
  \limsup_{\lambda \to \infty} I_3(t) = - \frac{1}{2}\int_{\mathbb{R}^d} (K+|x|^2)^{\nu_2/2} |\tilde{q}(t,x)|^2\exp\left( -t (K+|x|^2)^{\nu_2/2} \right)dx.
\end{equation}
By \eqref{eq:F-bound}
\begin{equation}
  \limsup_{\lambda \to \infty}I_2(t) \leq \int_{\mathbb{R}^d} (C(1 + |x|^{\nu_1}) + \nu_2^2(K+|x|^2)^{\nu_2 -1}) |\tilde{q}(t,x)|^2 \exp\left( -t (K+|x|^2)^{\nu_2/2}\right)dx.
\end{equation}
By first integrating \eqref{eq:weighted-q-lambda-pde} in time and then taking the limit as $\lambda \to +\infty$ we can conclude that for a fixed $T>0$
\begin{align}
  &\int_{\mathbb{R}^d} |\tilde{q}(T,x)|^2 \exp \left(- T (K+|x|^2)^{\frac{\nu_2}{2}} \right)dx \nonumber\\
  &\leq \int_0^T \int_{\mathbb{R}^d} \Bigg( C(1 + |x|^{\nu_1}) + \nu_2^2(1 + t+ t^2/2)(K+|x|^2)^{\nu_2 -1} \nonumber\\
  &\qquad\qquad\qquad  + t\nu(\nu+d-2)(K+|x|^2)^{\frac{\nu_2}{2} -1}  \nonumber\\
   &\qquad\qquad\qquad - \frac{1}{2}(K+|x|^2)^{\frac{\nu_2}{2}}\Bigg)|\tilde{q}(t,x)|^2 \exp\left( -t(K+|x|^2)^{\frac{\nu_2}{2}} \right) dxdt.
\end{align}
Recall that $\nu_1<\nu_2<2$. In particular, this means that $\nu_2-1 < \frac{\nu_2}{2}$ and $\frac{\nu_2}{2}-1 < \frac{\nu_2}{2}$. Therefore, if $K>0$ is chosen large enough (in a way that depends on $M$, $C$ and $T$), then the right-hand side of the above display is non-positive. This is enough to guarantee that for any $t \in [0,T]$ \begin{equation}
  \int_{\mathbb{R}^d} |\tilde{q}(t,x)|^2 \exp \left( -t (K+|x|^2)^{\frac{\nu_2}{2}} \right)dx = 0.
\end{equation}
Because $T>0$ is arbitrary, this implies that $u_1(t,x) = u_2(t,x)$ for all $t>0$ and $x \in \mathbb{R}^d$.

\begin{appendix}
\section{Appendix: Approximation of weakly differentiable weighted functions}
%
Let $C_b(\mathbb{R}^d)$ denote the space of bounded continuous functions from $\mathbb{R}^d \to \mathbb{R}$. Let $G(t,x)$ be the heat kernel defined in \eqref{eq:heat-kernel}.
Define for $\lambda>0$ the resolvent operator $R(\lambda):C_b(\mathbb{R}^d) \to C_b(\mathbb{R}^d)$
\begin{equation}
  R(\lambda) \phi(x) = \int_0^\infty e^{-\lambda r} \int_{\mathbb{R}^d} G(r,x-y)\phi(y)dydr.
\end{equation}
Because $\phi$ is assumed to be uniformly bounded, $R(\lambda)$ is a bounded linear operator with norm $\frac{1}{\lambda}$. Furthermore, $R(\lambda)\phi$ is twice differentiable and
\begin{equation}
  \frac{1}{2}\Delta R(\lambda) \phi(x) = \lambda R(\lambda)\phi(x) - \phi(x).
\end{equation}
In this way, $R(\lambda) = \left(\lambda - \frac{1}{2}\Delta \right)^{-1}$.

Unlike in classical Hille-Yosida theory (see \cite[Chapter 1.3]{pazy}), the domain of $\frac{1}{2}\Delta$ is not dense in $C_b(\mathbb{R}^d)$. Despite this, we still have a weak Hille-Yosida theory in the sense of Cerrai \cite{c-1994}.
For any compact set $K \subset \mathbb{R}^d$,
\begin{equation}
  \lim_{\lambda \to +\infty} \sup_{x \in K}|\lambda R(\lambda) \phi(x) - \phi(x)| =0,
\end{equation}
but convergence is not necessarily uniform over the whole space. See \cite{c-1994} for an example. This is due to the fact that $S(t): C_b(\mathbb{R}^d) \to C_b(\mathbb{R}^d)$ defined by $S(t)\phi(x) = \int_{\mathbb{R}^d} G(t,x-y)\phi(y)dy$ is a weakly continuous semigroup, but not a $C_0$ semigroup.

The following proposition shows that we can approximate weak solutions to PDEs by strong solutions to PDEs.
\begin{proposition} \label{prop:WLOG-strong}
Assume that $v(t,x)$ is a bounded weak solution to the PDE
\begin{equation} \label{eq:weak-PDE}
  \frac{\partial v}{\partial t}(t,x) = \frac{1}{2}\Delta v(t,x) + \varphi(t,x,v(t,x))
\end{equation}
where $\varphi$ is uniformly bounded and Lipschitz continuous in its third variable in the sense that there exists $L>0$ such that
\begin{equation}
  \sup_{t \geq 0} \sup_{x \in \mathbb{R}^d} |\varphi(t,x,v_1) - \varphi(t,x,v_2)| \leq L|v_1 - v_2|.
\end{equation}
Then there exists a sequence $v_\lambda(t,x)$ of strong solutions to the PDE
\begin{equation} \label{eq:strong-PDE}
  \frac{\partial v_\lambda}{\partial t}(t,x) = \frac{1}{2}\Delta v_\lambda(t,x) + \varphi(t,x,v_\lambda(t,x)) + \delta_\lambda(t,x)
\end{equation}
where for any compact $K \in \mathbb{R}^d$,
\begin{equation}
  \lim_{\lambda \to \infty} \sup_{x \in K} |v_\lambda(t,x) - v(t,x)| = 0
\end{equation}
and the remainder
\begin{equation}
  \lim_{\lambda \to \infty} \sup_{x \in K} |\delta_\lambda(t,x)| = 0.
\end{equation}
\end{proposition}

\begin{proof}
  Define $v_\lambda(t,x) = \lambda R(\lambda)v(t,x)$. Then $v_\lambda$ is strongly differentiable and strongly solves
  \begin{equation} \label{eq:strong v}
    \frac{\partial v_\lambda}{\partial t}(t,x) = \frac{1}{2}\Delta v_\lambda(t,x) + \varphi(t,x,v_\lambda(t,x)) + \delta_\lambda(t,x).
  \end{equation}
  In the above expression
  \begin{equation}
    \delta_\lambda(t,x) = \lambda R(\lambda) \varphi(t,\cdot,v(t,\cdot))(x) - \varphi(t,x,v_\lambda(t,x)).
  \end{equation}
  Because of the convergence properties of $\lambda R(\lambda)$
  and because of the Lipschitz continuity of $\varphi$, for any compact set $K \subset \mathbb{R}^d$,
  \begin{equation}
    \lim_{\lambda \to \infty}\sup_{x \in K} |\delta_\lambda(t,x)| = 0.
  \end{equation}
  Furthermore, $\delta_\lambda$ is uniformly bounded because $\varphi$ is uniformly bounded.
\end{proof}

The biggest disadvantage of Proposition \ref{prop:WLOG-strong} is that the convergence is only uniform over compact subsets of $\mathbb{R}^d$. If we weight our solutions, however, then the convergence becomes uniform.

\begin{corollary} \label{cor:WLOG-strong}
  Let $v(t,x)$ be  a bounded weak solution to \eqref{eq:weak-PDE}.
  Let $\rho: [0,+\infty) \times \mathbb{R}^d \to [1,+\infty)$ be a twice differentiable weight function satisfying $\lim_{|x| \to \infty} \rho(t,x) = +\infty$ and define the quotient
  \begin{equation}
    r(t,x) : = \frac{ v(t,x)}{\rho(t,x)}.
  \end{equation}
  There exists a sequence $r_\lambda(t,x)$ of strongly differentiable processes solving the PDE
  \begin{align}\label{eq:strong-quotient-PDE}
    \frac{\partial r_\lambda}{\partial t}(t,x) = &\frac{1}{2}\Delta r_\lambda(t,x) + \nabla r_\lambda(t,x) \cdot \frac{\nabla \rho(t,x)}{\rho(t,x)} + \frac{1}{2}\frac{r_\lambda(t,x) \Delta \rho(t,x)}{\rho(t,x)}\nonumber\\
    &+ \frac{ \varphi(t,x, r_\lambda(t,x) \rho(t,x)) }{\rho(t,x)} - \frac{r_\lambda(t,x) \frac{\partial \rho}{\partial t}(t,x)}{\rho(t,x)} + \frac{\delta_\lambda(t,x)}{\rho(t,x)}.
  \end{align}
  such that
  \begin{equation}
    \lim_{\lambda \to \infty} \sup_{x \in \mathbb{R}^d}|r_\lambda(t,x) - r(t,x)| = 0
  \end{equation}
  and
  \begin{equation}
    \lim_{\lambda \to \infty} \sup_{x \in \mathbb{R}^d} \left| \frac{\delta_\lambda(t,x)}{\rho(t,x)} \right| = 0.
  \end{equation}
\end{corollary}

\begin{proof}
Let $v_\lambda$ be from Proposition \ref{prop:WLOG-strong} and define the quotient
\begin{equation}
  r_\lambda(t,x): = \frac{v_\lambda(t,x)}{\rho(t,x)}.
\end{equation}
Because of \eqref{eq:strong v}, $r_\lambda$ strongly solves
\begin{align}
  \frac{\partial r_\lambda}{\partial t}(t,x) = \frac{1}{2}\frac{\Delta v_\lambda(t,x)}{\rho(t,x)} + \frac{\phi(t,x,v_\lambda(t,x))}{\rho(t,x)} - \frac{v_\lambda(t,x) \frac{\partial \rho}{\partial t}(t,x)}{(\rho(t,x))^2} + \frac{\delta_\lambda(t,x)}{\rho(t,x)}.
\end{align}

Then because $v_\lambda (t,x) = r_\lambda(t,x)\rho(t,x)$ and
\begin{equation}
  \frac{1}{2}\frac{\Delta v_\lambda(t,x)}{\rho(t,x)} = \frac{1}{2}\Delta r_\lambda(t,x) + \nabla r_\lambda(t,x) \cdot \frac{\nabla \rho(t,x)}{\rho(t,x)} + \frac{1}{2}\frac{r_\lambda(t,x) \Delta \rho(t,x)}{\rho(t,x)}
\end{equation}
it follows that $r_\lambda$ strongly solves
\begin{align}
  \frac{\partial r_\lambda}{\partial t}(t,x) = &\frac{1}{2}\Delta r_\lambda(t,x) + \nabla r_\lambda(t,x) \cdot \frac{\nabla \rho(t,x)}{\rho(t,x)} + \frac{1}{2} \frac{r_\lambda(t,x) \Delta \rho(t,x)}{\rho(t,x)}\nonumber\\
   &+ \frac{ \varphi(t,x, r_\lambda(t,x) \rho(t,x)) }{\rho(t,x)} - \frac{r_\lambda(t,x) \frac{\partial \rho}{\partial t}(t,x)}{\rho(t,x)} + \frac{\delta_\lambda(t,x)}{\rho(t,x)}.
\end{align}

Because $\delta_\lambda(t,x)$ is uniformly bounded and converges to zero uniformly on compact sets and because $\rho(t,x)$ converges to $\infty$
the remainder
\begin{equation}
  \lim_{\lambda \to \infty}\sup_{x \in \mathbb{R}^d}\frac{|\delta_\lambda(t,x)|}{\rho(t,x)} = 0
\end{equation}
for any $t>0$.

Similarly, $v_\lambda(t,x)$ converges to $v(t,x)$ uniformly over compact sets and are uniformly bounded. Therefore,
\begin{equation}
  \lim_{\lambda \to \infty}\sup_{x \in \mathbb{R}^d}|r_\lambda(t,x) - r(t,x)| = 0.
\end{equation}
\end{proof}

\section{Upper-left derivative of a supremum}
Define the upper-left derivative of a real-valued function $y: \mathbb{R} \to \mathbb{R}$ by
\begin{equation}
  \frac{d^-}{dt} y(t) := \limsup_{h \downarrow 0} \frac{y(t) - y(t-h)}{h}.
\end{equation}

Recall the definition of the function space $C_0$ from \eqref{eq:C_0-def}.

The next proposition is a generalization of Proposition D.4 in the appendix of \cite{dpz-book}.
\begin{proposition} \label{prop:norm-diff}
  Assume that $v: [0,+\infty)\times \mathbb{R}^d \to \mathbb{R}$ is differentiable in the first variable and that $v(t,\cdot) \in C_0(\mathbb{R}^d)$ for all $t>0$, then the real-valued function $t \mapsto \sup_{x \in \mathbb{R}^d} v(t,x)$ has a bounded upper-left-derivative  whenever $\sup_{x \in \mathbb{R}^d} v(t,x)>0$. Furthermore, for any maximizer $x_t \in \mathbb{R}^d$ such that
  \[v(t,x_t) = \sup_{x \in \mathbb{R}^d} v(t,x), \]
  \begin{equation} \label{eq:norm-diff}
    \frac{d^-}{dt} \sup_{x \in \mathbb{R}^d} v(t,x) \leq \frac{\partial v}{\partial t}(t,x_t)
  \end{equation}
\end{proposition}

\begin{proof}
  Let $t>0$ and assume that $\sup_{x \in \mathbb{R}^d} v(t,x)>0$. Because $v(t,\cdot) \in C_0(\mathbb{R}^d)$, a maximum is attained at at least one $x_t \in \mathbb{R}^d$.

  For any $h\in (0,t)$, it trivially holds that
  \begin{equation}
    v(t-h, x_t) \leq \sup_{x \in \mathbb{R}^d} v(t-h,x).
  \end{equation}
  Then by the definition of upper-left-derivative,
  \begin{align}
    &\frac{d^-}{dt} \sup_{x \in \mathbb{R}^d} v(t,x) \nonumber\\
    &=  \limsup_{h \downarrow 0}\frac{\sup_{x \in \mathbb{R}^d} v(t,x) - \sup_{x \in \mathbb{R}^d} v(t-h,x)}{h}\nonumber\\
    & \leq \limsup_{h \downarrow 0} \frac{ v(t,x_t) - v(t-h,x_t)}{h}\nonumber\\
    & \leq \frac{\partial v}{\partial t} (t,x_t).
  \end{align}
\end{proof}

\end{appendix}

\bibliographystyle{amsplain}
\bibliography{super-linear}
\end{document}